\documentclass[12pt]{amsart}

\usepackage{amssymb, amsfonts, amsthm,float}
\usepackage{graphicx}

\newtheorem{theorem}{Theorem}[section]
\newtheorem{lemma}[theorem]{Lemma}

\newtheorem{proposition}[theorem]{Proposition}

\theoremstyle{definition}
\newtheorem{definition}[theorem]{Definition}

\newtheorem{question}[theorem]{Question}

\numberwithin{equation}{section}

\newcommand{\be}{\begin{equation}}
\newcommand{\ee}{\end{equation}}

\newcommand{\abs}[1]{\lvert #1\rvert}

\newcommand{\Rn}{{\mathbb R}^n}

\newcommand{\Rt}{{\mathbb R}^2}
\newcommand{\D}{{\mathbb D}}

\newcommand{\psubset}{\varsubsetneq}

 \renewcommand{\phi}{\varphi}

\newcounter{minutes}
\divide\time by 60
\newcounter{hours}
\multiply\time by 60 \addtocounter{minutes}{-\time}

\begin{document}
\def\thefootnote{}
\footnotetext{ \texttt{File:~\jobname .tex,
           printed: \number\year-\number\month-\number\day,
           \thehours.\ifnum\theminutes<10{0}\fi\theminutes}
} \makeatletter\def\thefootnote{\@arabic\c@footnote}\makeatother

\title[Geometric properties of $\varphi$-uniform domains]
{Geometric properties of $\varphi$-uniform domains}
\author[P. H\"ast\"o]{Peter H\"ast\"o}
\address{Department of Mathematical Sciences, P.O.~Box~3000,
90014~University of Oulu, Finland}
\email{peter.hasto@oulu.fi}

\author[R. Kl\'en]{Riku Kl\'en}
\address{Department of Mathematics and Statistics, University of Turku,
FIN-20014 Turku, Finland}
\email{riku.klen@utu.fi}

\author[S. K. Sahoo]{Swadesh Kumar Sahoo}
\address{Discipline of Mathematics, Indian Institute of Technology Indore,
Simrol, Khandwa Road, Indore-453 552, India}
\email{swadesh@iiti.ac.in}

\author[M. Vuorinen]{Matti Vuorinen}
\address{Department of Mathematics and Statistics, University of Turku,
FIN-20014 Turku, Finland}
\email{vuorinen@utu.fi}






\begin{abstract}
We consider proper subdomains $G$ of $\Rn$ and their images $G'=f(G)$
under quasiconformal mappings $f$ of $\Rn$. We compare the distance ratio metrics
of $G$ and $G'$; as an application we show that $\varphi$-uniform domains are preserved under quasiconformal
mappings of $\Rn$. A sufficient condition for $\varphi$-uniformity is
obtained in terms of the quasi-symmetry condition.
We give a geometric condition for uniformity:
If $G\subset\Rn$ is $\phi$-uniform and satisfies the twisted cone condition,
then it is uniform.
We also construct a planar $\phi$-uniform domain whose complement is not
$\psi$-uniform for any $\psi$.
\end{abstract}

\maketitle


\section{Introduction and Main Results}\label{intro}

Classes of subdomains of the Euclidean $n$-space $\mathbb{R}^n$, $n\ge 2$,
occur often in geometric function theory and modern mapping theory.
For instance, the boundary regularity of a conformal mapping
of the unit disk onto a domain $D$ depends on the
properties of $D$ at its boundary. Similar results have been established
for various classes of functions such as quasiconformal mappings and mappings
with finite distortion. In such applications, uniform domains and their
generalizations occur \cite{Ge99,GO79,gh,Ko09,Va-book,Va88,Va91,Va98,Vu-book};
$\phi$-uniform domains have been recently studied in \cite{KLSV14}.

Let $\gamma\colon [0,1]\to G\subset \Rn$ be a path, i.e.\  a continuous
function.
All the paths $\gamma$ 
are assumed to be rectifiable,
that is, to have finite Euclidean length (notation-wise we write $\ell(\gamma) < \infty$).

Let $G\psubset \Rn$ be a domain
and $x,y\in G$. We denote by $\delta_G(x)$, the Euclidean distance from $x$ to the boundary
$\partial G$ of $G$. When the domain is clear, we use
the notation $\delta(x)$.
The {\em $j_G$  metric} (also called the {\em distance ratio metric}) \cite{Vu85}
is defined by
$$ j_G(x,y) := \log\bigg( 1 + \frac{\abs{x-y} }{\delta(x) \wedge \delta(y)} \bigg),
$$
where $a \wedge b=\min\{a,b\}$.
A slightly different form of this metric was studied in \cite{GO79}.
The {\em quasihyperbolic metric} of $G$ is defined by
the quasihyperbolic-length-minimizing property
$$k_G(x,y)=\inf_{\gamma\in \Gamma(x,y)} \ell_k(\gamma),
\quad \ell_k(\gamma) =\int_\gamma \frac{|dz|}{\delta(z)}\,,
$$
where $\ell_k(\gamma)$ is the quasihyperbolic length of
$\gamma$ (cf. \cite{GP76}) and $\Gamma(x,y)$ is the set of all rectifiable curves joining $x$ and $y$ in $G$.
For a given pair of points $x,y\in G,$ the infimum is always
attained \cite{GO79}, i.e., there always exists a quasihyperbolic
geodesic $J_G[x,y]$ which minimizes the above integral,
$k_G(x,y)=\ell_k(J_G[x,y])\,$ and furthermore with the property that
the distance is additive on the geodesic: $k_G(x,y)=$
$k_G(x,z)+k_G(z,y) $ for all $z\in J_G[x,y]\,.$ It also satisfies the monotonicity property:
$k_{G_1}(x,y)\le k_{G_2}(x,y)$ for all $x,y\in G_2\subset G_1$. If the domain $G$ is
emphasized we call $J_G[x,y]$ a $k_G$-geodesic.
Note that for all domains $G$,
\begin{equation}\label{jlek}
  j_G(x,y) \le k_G(x,y)
\end{equation}
for all $x,y \in G$ \cite{GP76}.

In 1979, uniform domains were introduced by Martio and Sarvas \cite{MS79}.
A domain $G\subset \Rn$ is said to be {\em uniform} if there exists
$C\ge 1$ such that for each pair of points $x,y\in G$
there is a path $\gamma\subset G$ with
\begin{enumerate}
\item[(i)] $\ell(\gamma)\le C\,|x-y| $; and
\item[(ii)] $\delta(z)\ge \displaystyle\frac{1}{C} [\ell(\gamma[x,z])\wedge \ell(\gamma[z,y])]$ for all $z\in \gamma$.
\end{enumerate}
Subsequently, Gehring and Osgood \cite{GO79} characterized uniform domains in terms
of an upper bound for the quasihyperbolic metric as follows: a
domain $G$ is {\em uniform} if and only if there exists a constant $C\ge
1$ such that
$$k_G(x,y)\le C j_G(x,y)
$$
for all $x,y\in G$. As a matter of fact, the above inequality
appeared in \cite{GO79} in a form with an additive constant on the
right hand side; it was shown by Vuorinen \cite[2.50]{Vu85} that the
additive constant can be chosen to be $0$.  This observation leads to
the definition of $\varphi$-uniform domains introduced in \cite{Vu85}.

\begin{definition}\label{phiunif} 
Let $\varphi:\, [0,\infty)\to [0,\infty)$ be a strictly increasing homeomorphism
with $\varphi(0)=0$. A domain $G\psubset\Rn$ is said to be {\em $\varphi$-uniform} if
$$k_G(x,y)\le \varphi\left(\frac{|x-y|}{\delta(x) \wedge \delta(y)}\right)
$$
for all $x,y\in G$.
\end{definition}

An example of a $\varphi$-uniform domain which is not uniform is given
in Section~\ref{SecComplement}. That domain has the property that
its complement is not $\psi$-uniform for any $\psi$.

V\"ais\"al\"a has also investigated the class of
 $\varphi$-domains \cite{Va91} (see also \cite{Va98} and references therein)
and pointed out that $\varphi$-uniform domains are nothing but uniform
under the condition that $\varphi$ is a slow function, i.e.
$\varphi(t)/t \to 0$ as $t\to \infty$.

In the above definition, uniform domains are characterized by the quasi-convexity (i) and
twisted-cone (ii) conditions.
In Section~\ref{SecUinf}, we show that the former can be replaced
by $\phi$-uniformity, which may in some situations be easier to
establish.

\begin{theorem}\label{phi-unif+cone}
If a domain $G\psubset \Rn$ is $\varphi$-uniform and
satisfies twisted cone condition, then it is uniform.
\end{theorem}

Let $G\subset \Rn$ be a domain and $f:\,G \to f(G) \subset \Rn$ be a homeomorphism. The linear
dilatation of $f$ at $x\in G$ is defined by
$$H(f,x):=\limsup_{r\to 0}\frac{\sup\{|f(x)-f(y)|:\, |x-y|=r\}}{\inf\{|f(x)-f(z)|:\, |x-z|=r\}}\,.
$$
We adopt the definition of $K$-quasiconformality from V\"ais\"al\"a \cite{Va-book}. If $f$ is $K$-quasiconformal then
$\sup_{x\in G}H(f,x)\le c(n,K)<\infty$.

In Section~\ref{MainSec} we study $\phi$-uniform domains in relation
to quasiconformal and quasisymmetric mappings.
Gehring and Osgood \cite[Theorem~3 and Corollary~3]{GO79},  proved that uniform domains are invariant under quasiconformal mappings of $\Rn$. Our next theorem extends this result to the case of $\phi$-uniform domains.

\begin{theorem}\label{qc-inv-thm}
Suppose that $G\psubset \Rn$ is a $\varphi$-uniform domain and $f:\mathbb{R}^n \to \mathbb{R}^n$ is a
quasiconformal mapping which maps $G$ onto $G'\psubset \Rn$.
Then $G'$ is $\varphi_1$-uniform for some $\varphi_1$.
\end{theorem}


\section{Quasiconformal and quasi-symmetric mappings}\label{MainSec}

In general, quasiconformal mappings of a uniform domain do not map onto a uniform domain.
For example, by the Riemann Mapping Theorem, there exists a conformal mapping of the unit disk
$\D=\{z\in\Rt:\, |z|<1\}$ onto the simply connected domain $\D\setminus [0,1)$. Note that the unit
disk $\D$ is ($\varphi$-)uniform whereas the domain $\D\setminus [0,1)$ is not.
However, this changes if we consider quasiconformal mappings of the
whole space $\Rn$: uniform domains are
invariant under quasiconformal mappings of $\Rn$ \cite{GO79}.
In this section we provide the analogue for $\phi$-uniform domains.

We notice that the quasihyperbolic metric and the distance ratio metric
have similar natures in several senses. For instance,
if $f:\,\Rn \to \Rn$ is a M\"obius mapping that takes a domain onto another,
then $f$ is $2$-bilipschitz with respect to the quasihyperbolic metric \cite{GP76}.
Counterpart of this fact with respect to the distance ratio metric can be obtained
from the proof of \cite[Theorem~4]{GO79} with the bilipschitz constant $2$.

\begin{lemma}\label{GO-thm3}\cite[Theorem~3]{GO79}
For $n\ge 2$ and $K\ge 1$ there exists a constant $c$ depending only on $n$ and $K$ such that,
if $f:G\to G'$ is a $K$-quasiconformal mapping of $G\psubset\Rn$ onto $G'\psubset\Rn,$
then
$$k_{G'}(f(x),f(y))\le c\,\max\{k_G(x,y),k_G(x,y)^\alpha\}
$$
for all $x,y\in G$, where $\alpha=K^{1/(1-n)}$.
\end{lemma}

We obtain an analogue of Lemma~\ref{GO-thm3} for $j_G$ with the
help of the following result of Gehring's and Osgood's:

\begin{lemma}\label{GO-thm4}\cite[Theorem~4]{GO79}
For $n\ge 2$ and $K\ge 1$ there exist constants $c_1$ and $d_1$ depending only on $n$ and $K$ such that
if $f \colon \Rn \to \Rn$ is a $K$-quasiconformal mapping which maps $G\psubset \Rn$ onto $G'\psubset \Rn$, then
$$j_{G'}(f(x),f(y))\le c_1\,j_G(x,y)+d_1
$$
for all $x,y\in G$.
\end{lemma}

In order to investigate the quasiconformal invariance property of $\varphi$-uniform domains
we reformulate Lemma~\ref{GO-thm4} in the form of the following lemma.
We make use of both the above lemmas in the reformulation.

\begin{lemma}\label{j-qc-thm}
For $n\ge 2$ and $K\ge 1$ there exists a constant $C$ depending only on $n$ and $K$ such that
if $f \colon \Rn \to \Rn$ is a $K$-quasiconformal mapping which maps $G$ onto $G'$, then
$$j_{G'}(f(x),f(y))\le C\,\max\{j_G(x,y),j_G(x,y)^\alpha\}
$$
for all $x,y\in G$, where $\alpha=K^{1/(1-n)}$.
\end{lemma}

\begin{proof}
Without loss of generality we assume that $\delta(x)\le \delta(y)$ for $x,y\in G$.
Suppose first that $y\in G\setminus B^n (x,\delta(x)/2)$.
Since $|x-y|\ge \delta(x)/2$, it follows that $j_G(x,y)\ge \log (3/2)$.
By Lemma~\ref{GO-thm4}, we obtain
$$j_{G'}(f(x),f(y))\le \left(c_1+\frac{d_1}{\log (3/2)}\right)\, j_G(x,y)\,.
$$

Suppose then that $y\in B^n(x,\delta(x)/2)$.
By \cite[Lemma 3.7 (2)]{Vu-book}, $k_G(x,y)\le 2j_G(x,y)$.
Hence we obtain that
\begin{align*}
j_{G'}(f(x),f(y))
& \le k_{G'}(f(x),f(y))
 \le c\,\max\{k_G(x,y),k_G(x,y)^\alpha\}\\
& \le 2c \,\max\{j_G(x,y),j_G(x,y)^\alpha\}\,,
\end{align*}
where the first inequality always holds by \eqref{jlek} and the second inequality is due to Lemma~\ref{GO-thm3}.
%
\end{proof}

As a consequence of Lemmas~\ref{GO-thm3} and \ref{j-qc-thm},
we prove our main
result Theorem~\ref{qc-inv-thm} about the invariance property of $\varphi$-uniform
domains under quasiconformal mappings of $\Rn$.

\begin{proof}[Proof of Theorem~\ref{qc-inv-thm}]
By 
Lemma~\ref{j-qc-thm}, there exists a constant $C$ such that
\begin{equation}\label{qc-eq1}
j_G(x,y)\le C\,\max\{j_{G'}(f(x),f(y)),j_{G'}(f(x),f(y))^\alpha\}
\end{equation}
for all $x,y\in G$. Define $\psi(t):= \phi(e^t-1)$. Then
\begin{align*}
k_{G'}(f(x),f(y))
& \le  c\, \max\{k_G(x,y),k_G(x,y)^\alpha\}\\
& \le c\, \max\{\psi(j_G(x,y)),\psi(j_G(x,y))^\alpha\}\\
& \le c\, \max\{\psi(C\,\max\{j_{G'}(f(x),f(y)),j_{G'}(f(x),f(y))^\alpha\}),\\
& \hspace*{2cm} \psi(C\,\max\{j_{G'}(f(x),f(y)),j_{G'}(f(x),f(y))^\alpha\})^\alpha\}\,,
\end{align*}
where the first inequality is due to Lemma~\ref{GO-thm3}, the second inequality holds by hypothesis, and the last
inequality is due to (\ref{qc-eq1}).
Thus, $G'$ is $\varphi_1$-uniform with
$$\varphi_1(t)=c\max\{\psi(C\max\{\log(1+t),\log(1+t)^\alpha\}),
\psi(C\max\{\log(1+t),\log(1+t)^\alpha\})^\alpha\}\,,
$$
where $\alpha=K^{1/(1-n)}$.
\end{proof}

A mapping $f:\, (X_1,d_1)\to (X_2,d_2)$ is said to be $\eta$-quasi-symmetric ($\eta$-QS)
if there exists a strictly increasing homeomorphism $\eta:\,[0,\infty)\to [0,\infty)$ with $\eta(0)=0$
such that
$$\frac{d_2(f(x),f(y))}{d_2(f(y),f(z))}\le \eta\left(\frac{d_1(x,y)}{d_1(y,z)}\right)
$$
for all $x,y,z\in X_1$ with $x\neq y\neq z$. Here $(X_1,d_1)$ and $(X_2,d_2)$ are metric spaces.

Note that $L$-bilipschitz mappings are $\eta$-QS with $\eta(t)=L^2t$ and $\eta$-QS mappings have
the linear dilitation bounded by $\eta(1)$. It is pointed out in \cite{TV80} that quasiconformal mappings are
locally quasi-symmetric.

The following result gives a sufficient condition for $G$ to be a $\varphi$-uniform domain.

\begin{proposition}\label{hksv-qs-prop}
If the identity mapping $id:\, (G,j_G)\to (G,k_G)$ is $\eta$-QS, then $G$ is $\varphi$-uniform
for some $\varphi$ depending on $\eta$ only.
\end{proposition}

\begin{proof}
By hypothesis we have
$$\frac{k_G(x,y)}{k_G(y,z)}\le \eta\left(\frac{j_G(x,y)}{j_G(y,z)}\right)
$$
for all $x,y,z$ with $x\neq y\neq z$. Choose $z ~(\neq y)$ such that $\delta(z)=e^{-1}\delta(y)$.
Then
$$j_G(y,z)=k_G(y,z)=\log\left(1+\frac{|y-z|}{\delta(z)}\right)
=\log\left(\frac{\delta(y)}{\delta(z)}\right)=1\,.
$$
Hence, by the hypothesis we conclude that
$$k_G(x,y)\le \eta(j_G(x,y))\,.
$$
This shows that $G$ is $\varphi$-uniform
with $\varphi(t)=\eta(\log(1+t))$.
\end{proof}

\begin{question}
Is the converse of Proposition~\ref{hksv-qs-prop} true\,?
\end{question}



\section{The $\varphi$-uniform domains which are uniform}\label{SecUinf}

A domain $G\subset \Rn$ is said to satisfy the
{\em twisted cone condition},
if for every $x,y\in G$ there exists a rectifiable path $\gamma\subset G$
joining $x$ and $y$ such that
\begin{equation}\label{cone}
\min\{\ell(\gamma[x,z]),\ell(\gamma[z,y])\}\le c\,\delta(z)\quad \mbox{ for all $z\in\gamma$}
\end{equation}
and for some constant $c>0$. Sometimes we call the path $\gamma$ a {\em twisted path}.
Domains satisfying the twisted cone condition are also called John domains (see for instance \cite{GHM89,He99,KL98,NV91}).
If,  in addition,  $\ell(\gamma)\le c\,|x-y|$ holds then
the domain $G$ is uniform. Note that the path $\gamma$ in the
definition of the twisted cone condition
may be replaced by a quasihyperbolic geodesic (see \cite{GHM89}).

We observe from Section~1 that a $\varphi$-uniform domain need not be uniform (or quasi-convex).
Nevertheless, a $\phi$-uniform domain satisfying the
twisted-cone condition is uniform.

\begin{proof}[Proof of Theorem~\ref{phi-unif+cone}]
Assume that $G$ is $\varphi$-uniform and satisfies the twisted cone condition (\ref{cone}).
Let $x,y\in G$ be arbitrary and $\gamma$ be a twisted path in $G$ joining $x$ and $y$.
Choose $x',y'\in\gamma$ such that $\ell(\gamma[x,x'])= \ell(\gamma[y,y']) = \frac1{10}|x-y|$.

Now, by the cone condition, we have
$$\delta(x')\ge \frac{1}{c}\min\{\ell(\gamma[x,x']),\ell(\gamma[x',y])\}
~\mbox{ and }~
\delta(y')\ge \frac{1}{c}\min\{\ell(\gamma[x,y']),\ell(\gamma[y',y])\}.
$$
By the choice of $x'$ and $y'$, on one hand, we see that
$$
\ell(\gamma[x',y])\ge |x'-y|\ge |x-y|-|x-x'|\ge \tfrac9{10} |x-y|.
$$
On the other hand, $\ell(\gamma[x,x'])=\frac1{10}|x-y|$.
The same holds for $x$ and $y$ interchanged. Thus,
\begin{equation}\label{thm3.4-eqn1}
\min\{\delta(x'),\delta(y')\}\ge \tfrac1{10c}|x-y|.
\end{equation}

To complete the proof, our idea is to prove the following three inequalities:
\begin{equation}\label{3-eqns}
\left\{\begin{array}{lll}
k_G(x,x') \le a_1\,j_G(x,x')\le b_1\,j_G(x,y);\\[2mm]
k_G(x',y') \le b_2\,j_G(x,y);\\[2mm]
k_G(y',y) \le a_3\,j_G(y',y)\le b_3\,j_G(x,y)
\end{array}\right.
\end{equation}
for some constants $a_i,b_i$, $i=1,2,3$.
Finally, the inequality
$$k_G(x,y)\le k_G(x,x')+k_G(x',y')+k_G(y',y)\le c\,j_G(x,y)
$$
with $(c=b_1+b_2+b_3)$ would conclude the proof of the theorem.
It is sufficient to show the first two lines in \eqref{3-eqns}, as the third is analogous
to the first.

We start with a general observation:
if $j_G(z,w)< \log \frac 32$, then $z$ lies in the ball $B(w, \frac12 \delta(w))$, and we
can connect the points by the segment $[z,w]\subset G$. Furthermore, due to
\cite[Lemma~3.7~(2)]{Vu-book} $k_G(z,w)\le 2\,j_G(z,w)$.
Thus in each inequality between the $k$ and $j$ metrics, we may assume that
$j_G(z,w)\ge \log \frac 32$ for all $z,w\in G$.

First we prove the second line of (\ref{3-eqns}). Since
$G$ is $\varphi$-uniform and $\varphi$ is an increasing homeomorphism,
\begin{equation*}
k_G(x',y')\le \varphi\left(\frac{|x'-y'|}{\min\{\delta(x'), \delta(y')\}}\right)\le \varphi(12 c),
\end{equation*}
where the triangle inequality $|x'-y'|\le |x'-x|+|x-y|+|y-y'|$ and the relation \eqref{thm3.4-eqn1} are
applied to obtain the second inequality. On the other hand, $j_G(x,y)\ge \log \frac32$. Thus
$$
k_G(x',y')\le b_2\,j_G(x,y)
$$
with $b_2=\varphi(12c)/\log \frac32$.

Then we consider the first line of \eqref{3-eqns}:
$k_G(x,x') \le a_1\,j_G(x,x')\le b_1\,j_G(x,y)$.
The second inequality is easy to prove. Indeed, we have
\begin{align*}
j_G(x,x')
 = \log\left(1+\frac{|x-x'|}{\min\{\delta(x),\delta(x')\}}\right)
 < \log\left(1+\frac{(1+c)|x-y|}{\delta(x)}\right)
 \le (1+c)\,j_G(x,y),
\end{align*}
where the first inequality holds since $|x-x'|\le \frac1{10}|x-y|$ and
$\min\{\delta(x),\delta(x')\}\ge \delta(x)/(1+c)$.

Fix a point $z\in \gamma$ with $\ell(\gamma[x,z]) = \frac12 \delta(x)$ and denote
$\gamma_1=\gamma[x,z]$. Assume for the time being that $x'\not \in \gamma_1$
Clearly, $k_G(x,x') \le k_G(\gamma_1) + k_G(\gamma_2)$, where $\gamma_2=\gamma[z,x']$.
For $w\in \gamma_1$ we have $\delta(w)\ge \frac12 \delta(x)$, and for $w\in \gamma_2$,
the twisted cone condition and the fact $\ell(\gamma[w,y])\ge 9\ell(\gamma[x,w])$ together give
$\delta(w)\ge (1/c)\, \ell(\gamma[x,w])$. Thus we find that
\[
k_G(\gamma_1) \le \frac{\ell(\gamma_1)}{\tfrac12 \delta(x)} = 1
\le
\frac 1{\log \frac 32} j_G(x,x')\le \frac 1{\log \frac 32}(1+c)j_G(x,y).
\]
Furthermore,
\[
k_G(\gamma_2)
\le \int_{\ell(\gamma[x,z])}^{\ell(\gamma[x,x'])} c\frac {dt}{t}
= c \log\frac{\ell(\gamma[x,x'])}{\ell(\gamma[x,z])}
= c \log\frac{\frac1{10}|x-y|}{\frac12 \delta(x)}
\le c  j_G(x,y).
\]
So the inequality is proved in this case.
If, on the other hand, $x' \in \gamma_1$, then we set $z=x'$ and repeat the argument
of this paragraph for $\gamma_1$, since $\gamma_2$ is empty in this case.

This completes the proof of our theorem.
\end{proof}


\section{Complement of $\varphi$-uniform domains}\label{SecComplement}

In \cite[Section~3]{KLSV14} the following question was posed:
Are there any bounded planar $\varphi$-uniform domains
whose complementary domains are not $\varphi$-uniform?
In this section we show that the answer is ``yes''.

\begin{proposition}\label{prop:new}
There exists a bounded $\varphi$-uniform Jordan domain $D\psubset \Rt$ such that
$\Rt\setminus\overline{D}$ is not $\psi$-uniform for any $\psi$.
%
\end{proposition}

\begin{proof}
Fix $0<u<t< v < 1$.  Let
$R_k = (x_k, x_k+u^k)\times [0, v^k]$ be the rectangle,
$k \ge 1$. The parameter $x_k$ is chosen such that $x_1=0$ and
$x_{k+1}=x_k + u^k + t^k$. At the top of each
rectangle $R_k$ we place a semi-disc $C_k$ with radius $u^k/2$ and center on the midpoint of the
top side of $R_k$.
Set $s=u/(1-u)+t/(1-t)$.
With these elements we define $D$, shown in Figure~\ref{hksv-fig2new},
by
$$
  D:=((0,s)\times (-2,0))\cup (\cup_{k}R_k) \cup (\cup_{k}C_k).
$$
\begin{figure}[H]
\centering
  \includegraphics{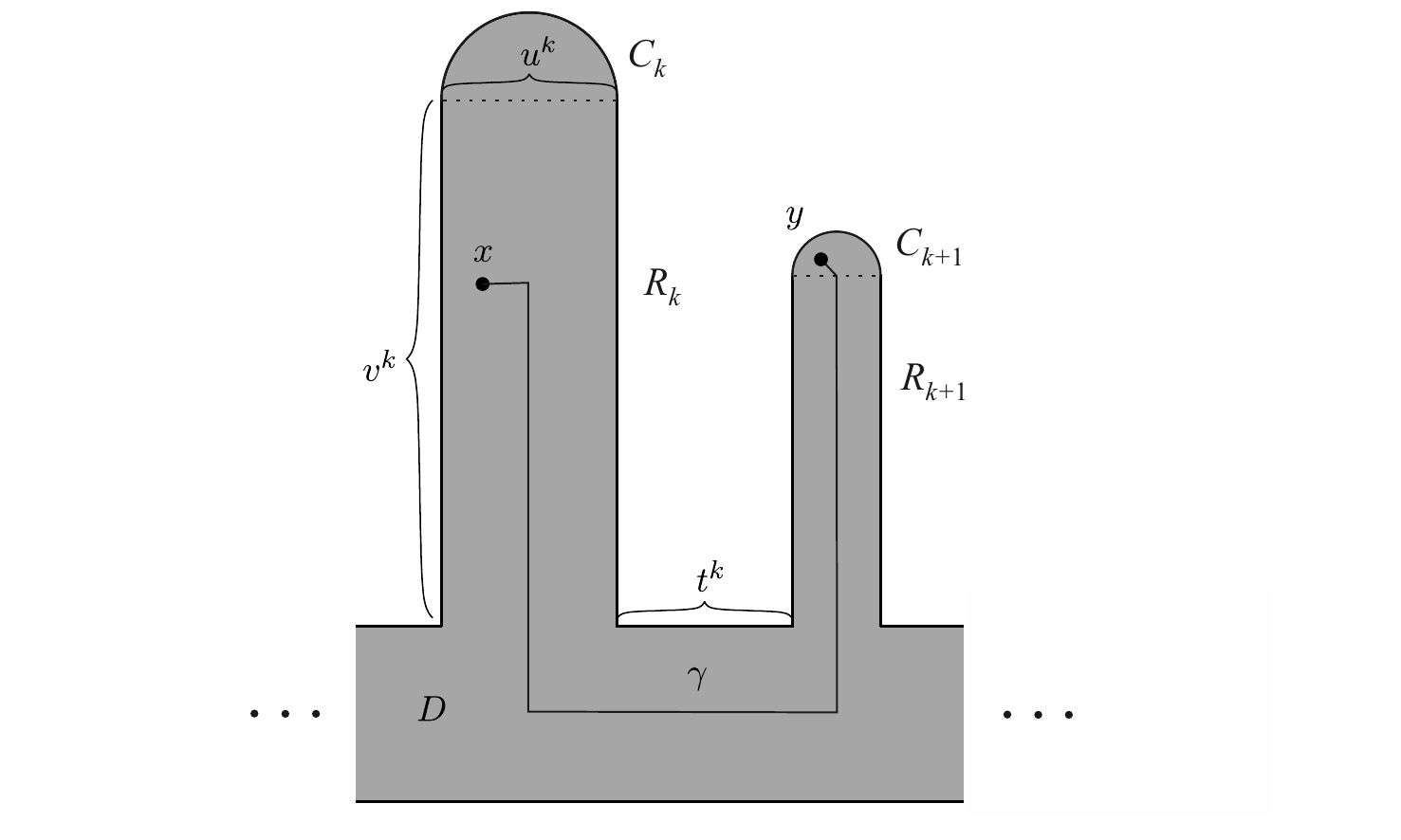}
\caption{The $\varphi$-uniform domain $D$ constructed in the proof of Proposition \ref{prop:new}.\label{hksv-fig2new}}
\end{figure}

Let us show that $D$ is $\phi$-uniform.
For $x\in R_k$ and $y\in R_l$, $k>l$, let
$d:=\min\{\delta(x),\delta(y)\}$. We choose a polygonal path
$\gamma$ as follows: from $x$ the shortest line segment to the
medial axis of
$R_k$, then horizontally at $y=-d$ and finally
from the medial axis of $R_l$ to $y$ along the shortest
line segment (see again Figure~\ref{hksv-fig2new}).
The lengths of the vertical and horizontal parts are at most
$2v^k+d$, $2v^l+d$ and $|x-y|+\frac{u^k+u^l}2$. The line segments
joining $x$ and $y$ to the medial axis have length at most
$u^k/2$ and $v^k/2$. The whole curve is at distance at least
$d$ from the boundary.
Thus
\[
k_D(x,y) \le k_D(\gamma) \le \frac{\ell(\gamma)}d \le
\frac{3v^k+3v^l+2d+|x-y|}{d}
\le \frac{8v^k + |x-y|}{d} .
\]
On the other hand,
\[
\frac{|x-y|}{\delta(x)\wedge \delta(y)} \ge \frac{t^k}{d} \,.
\]
Let $\phi(\tau):= \tau + 8 \tau^\alpha$, where $\alpha$ is such that
$t^\alpha = v$. Then $k_D(x,y) \le \phi(\frac{|x-y|}{\delta(x)\wedge \delta(y)})$. The case when $x,y\in R_k$ or in the base rectangle are
handled similarly, although they are simpler. Thus we conclude
that $D$ is $\phi$-uniform.

We show then that $\Rt\setminus\overline D$ is not $\psi$-uniform
for any $\psi$. We choose $z_k = (x_{k+1}-t^k/2, t^k)$ in the gap
between $R_k$ and $R_{k+1}$. Then
\[
\frac{|z_k-z_{k+1}|}{\delta(z_k)\wedge \delta(z_{k+1})} \le
\frac{t^k/2+u^{k+1}+t^{k+1}/2 + t^k}{t^{k+1}}
=
\frac 3{2t} + \frac 12 + \big(\tfrac u t\big)^{k+1}\le \frac 3{2t}
+\frac 32  .
\]
On the other hand, a curve connecting these points has length at least
$v^k-t^k/2$, so that
\[
k_{\Rt\setminus\overline D}(z_k, z_{k+1})
\ge \int_{t^k/2}^{v^k-t^k/2} \frac{dx}x
= \log \frac{v^k-t^k/2}{t^k/2} \to \infty
\]
as $k\to\infty$. Therefore, it is not possible to find
$\psi$ such that $k_{\Rt\setminus\overline D}(z_k, z_{k+1}) \le \psi(\frac{|z_k-z_{k+1}|}{\delta(z_k)\wedge \delta(z_{k+1})})$,
as claimed.
\end{proof}

\end{document}